\documentclass[12pt,reqno]{amsart}
\usepackage[comma,numbers]{natbib}
\usepackage[a4paper,left=1in,right=1in,top=1.25in,bottom=1.25in]{geometry}
\usepackage{color}
\usepackage{graphicx}

\newtheorem{theorem}{Theorem}[section]
\newtheorem{corollary}[theorem]{Corollary}
\newtheorem{lemma}[theorem]{Lemma}
\theoremstyle{definition}
\theoremstyle{remark}

\def\numberlikeadb{\global\def\theequation{\thesection.\arabic{equation}}}
\numberlikeadb

\newcommand{\eqa}{\begin{eqnarray}}
\newcommand{\ena}{\end{eqnarray}}
\newcommand{\eq}{\begin{equation}}
\newcommand{\en}{\end{equation}}
\newcommand{\eqs}{\begin{eqnarray*}}
\newcommand{\ens}{\end{eqnarray*}}

\def\Eq{\ =\ }
\def\Def{\ :=\ }
\def\Le{\ \le\ }
\def\pr{\mathbb{P}}
\def\ex{\mathbb{E}}
\def\ignore#1{}
\def\BB{\mathcal{B}}
\def\Z{\mathbb{Z}}
\def\e{\varepsilon}
\def\a{\alpha}
\def\tX{W}
\def\Blm{\left|}
\def\Brm{\right|}
\def\Bl{\left(}
\def\Br{\right)}
\def\Blb{\left\{}
\def\Brb{\right\}}
\def\s{\sigma}
\def\sjn{\sum_{j=1}^n}
\def\sn{\sum_{i=1}^n}
\def\bone{{\bf 1}}

\def\Ref#1{(\ref{#1})}

\def\ww{\mathcal{W}}
\def\bbb{{H}}
\def\bI{\mathbb{I}}

\def\re{\mathbb{R}}
\def\non{\nonumber}
\def\nin{\noindent}
\def\var{{\rm Var\,}}
\def\h{\eta}
\def\giv{\ |\ }
\def\ff{{\mathcal F}}
\def\xx{{\mathcal X}}

\def\um{^{(m)}}
\def\bone{{\bf 1}}

\def\bp{{\overline p}}
\def\bX{{\overline X}}
\def\bW{{\overline W}}
\def\ba{{\bar a}}

\def\th{\theta}
\def\ps{\psi}

\def\d{\delta}
\def\l{\lambda}

\def\g{\gamma}
\def\hps{{\hat\psi}}
\def\b{\beta}
\def\slo{{\sum_{l\ge0}}}
\def\t{\tau}
\def\atrate{\quad\mbox{at rate}\quad}
\def\Bi{{\rm Bi}\,}
\def\Br{\right)}

\begin{document}
\title[]{Connecting deterministic and stochastic metapopulation models}
\maketitle
\noindent A.D. BARBOUR, R. McVINISH and P.K. POLLETT  \footnote{ADB is supported in part by
Australian Research Council (Discovery Grants DP120102728 and DP120102398). PKP and RM are supported in part by the Australian
Research Council (Discovery Grant DP120102398 and the Centre of
Excellence for Frontiers in Mathematics and Statistics)}\\
Universit\"at Z\"urich and University of Queensland\\


\noindent ABSTRACT.  In this paper, we study the relationship between certain stochastic and deterministic versions of Hanski's incidence function model and the spatially realistic Levins model. We show that the stochastic version can be well approximated in a certain sense by the deterministic version when the number of habitat patches is large, provided that the presence or absence of individuals in a given patch is influenced by a large number of other patches. Explicit bounds on the deviation between the stochastic and deterministic models are given.


\noindent\emph{Key words}: Stochastic patch occupancy model (SPOM); Vapnik--Chervonenkis theory \\
\noindent\emph{MSC 2010}:  92D40; 60J10; 60J27

\section{Introduction}\label{intro}

Hanski's incidence function model \citep{Hanski:94} is perhaps the most widely used and studied metapopulation 
model in ecology. It is a discrete time Markov chain model, whose transition probabilities incorporate 
properties of the landscape to provide a realistic model of metapopulation dynamics. Numerous modifications, 
extensions and applications have been reported in the literature. In particular, we note \citet{AM:02}, who 
proposed a continuous time version. As these metapopulation models are finite state Markov chains, many 
quantities of interest can be calculated numerically, including the expected time to extinction and the 
quasi-stationary distribution. However, this does not aid our understanding of the model in general. 

Deterministic metapopulation models are often easier to analyse, allowing conditions for persistence to be determined fairly explicitly. For example, \citet{OH:01} made a detailed analysis of the spatially realistic Levins model \citep{HG:97}, providing, among other things, approximations of the equilibrium state and threshold conditions \citep[see also][]{OH:02}. However, these deterministic models expressed in terms of continuous quantities are only relevant insofar as they reflect properties of a related discrete stochastic model, and our primary interest here is in the extent to which this is true. Approximating Markov chains by deterministic processes is not a new idea, and results quantifying the approximation error have been obtained for a large class of models \citep[see][and references therein]{DN:08}; the stochastic metapopulation models that we are interested in do not fall into this class. 

In this paper, we show that, if the presence or absence of individuals in a given patch is evenly influenced by many other patches, the stochastic metapopulation models proposed in \citet{Hanski:94} and \citet{AM:02} are well approximated by the deterministic models in \citet{OH:01}. In Section~\ref{stoch-det}, we review these models, and describe how we measure the closeness of the deterministic model to the stochastic model. The parts of Vapnik--Chervonenkis theory needed for understanding this measure of closeness are briefly summarised. In Section~\ref{discrete}, we analyse the incidence function model, and establish two bounds on the difference between the outcomes of the deterministic and stochastic models.  Our first bound, given in Theorem~\ref{Thm1}, is simpler to derive than the second, Theorem~\ref{ADB-refined-Thm}, which is, however, usually asymptotically sharper; but neither bound in general dominates the other.  In Section~\ref{continuous}, we prove the corresponding bounds for the spatially realistic Levins model, in Theorem~\ref{ADB-Levins-Thm}. The proofs follow an approach used in \citet{BL:08}. We first construct a new metapopulation model where, conditional on the environmental 
variables, the patches are independent of each other. This independent patches metapopulation is well approximated by the deterministic model. We then couple the independent patches metapopulation to the original metapopulation and show that they remain close over finite time intervals. The paper concludes with some discussion. In particular, it is noted that the deterministic models are {\em not} shown to give good approximations to the analogous stochastic models, unless the presence or absence of individuals in a given patch is influenced by a large number of other patches, and that the approximation may otherwise be very poor. The example of recolonization only from immediately neighbouring patches in a metapopulation consisting of~$n$ patches arranged in line is enough to illustrate this.

\section{Stochastic and deterministic metapopulation models}\label{stoch-det}

\subsection{Incidence function model}\label{incidence-model}

The incidence function model of \citet{Hanski:94} for a metapopulation comprising $n $ patches is a 
discrete-time Markov chain on $\xx := \{0,1\}^{n} $. Denote this Markov chain by $ X_{t}= (X_{1,t},\ldots,X_{n,t}) $, where $ X_{i,t} = 1 $ if patch~$i$ is occupied at time $ t $ and $ X_{i,t} = 0 $ otherwise. In the generalization of the incidence function model considered here, patch $ i $ is described by two variables; its location $ z_{i} \in \re^d$ and a weight $ a_{i} > 0 $ which may be interpreted as the size of the patch. Other variables determining patch quality could be incorporated without changing the analysis. Writing $\ww :=\re^d \times \re_+$, let $ \sigma $ denote the set of vectors $ \{(z_{i},a_{i}) ,\, 1\le i\le n\} \subset \ww$; throughout, we let $\pr$ and~$\ex$ denote probability and expectation given~$\s$, and $ \mathbb{I}[\cdot] $ denote the indicator function taking the value 1 if the statement in $ [\cdot] $ is true and 0 otherwise. The transition probabilities of the Markov chain are determined by how well the patches are connected to each other and by the 
probability of local extinction.  Define the function $S_{i}: [0,1]^{n} \mapsto [0,\infty) $ by
\begin{equation}
   S_{i}(x) \Eq n^{-1}\sum_{j\neq i} x_{j}  a_{j} s_{ji},  \label{IFM:Connect}
\end{equation}
where $ s_{ji} = s_{ij} \ge 0$ for all $1\le i \ne j \le n$ and $s_{jj} := 0$, $1\le j\le n$; typically, for some $\a > 0$, 
\[
    s_{ji} \Def \exp(-\alpha \|z_{j} - z_{i}\|),\quad 1 \le j\neq i \le n.
\]
The connectivity measure of patch $ i $ at time $ t $ is given by $ S_{i}(X_{t}) $. Other forms such as those discussed in  \citet{Shaw:94} and \citet{MH:98} are also covered by our results. For bounded functions $ f_{C,i},f_{E,i} \colon [0,\infty) \to [0,\infty)$,  write $ C_{i}(x) = f_{C,i}(S_{i}(x))$ and $E_{i}(x) = f_{E,i}(S_{i}(x))$, $1\le i\le n$, $x \in [0,1]^{n}$. For any $m > 0$ such that $ m^{-1} \max\{C_i(x),E_i(x)\} \le 1$ for all $i$ and~$x$, define a Markov chain~$X\um $ such that, conditional on $ \left(X_{t}\um, \sigma\right)$, the $ X_{i,t+1}\um \ (i=1,\ldots,n) $ are independent with transition probabilities
\begin{equation}
   \pr\left(X_{i,t+1}\um=1 \ \middle| \ X_{t}\um\right)  \Eq 
       m^{-1} C_{i}(X_{t}\um)\left(1-X_{i,t}\um\right) +  (1-m ^{-1} E_{i}(X_t\um))X_{i,t}\um. \label{Eq1}
\end{equation}
If patch $ i $ is occupied at time $ t $, then that population survives to time $ t+1 $ with probability 
$ 1 - m ^{-1} E_{i}(X_t\um) $. Otherwise, it is colonised with probability $  m^{-1} C_{i}(X_{t}\um)$. 
This formulation of the colonisation and extinction probabilities is sufficiently flexible to cover many 
extensions of Hanski's incidence function model \citep{Hanski:94}, such as the inclusion of a rescue 
effect \citep{BKB:77,HMG:96}, the form of colonisation probabilities proposed by \citet{MN:02} and phase 
structure \citep{DP:95}.

For compatibility with the continuous time models that follow, the quantities $C_i(X)$ and~$E_i(X)$ should
be thought of as {\it rates\/} per unit time, and~$m^{-1}$ as a length of time, their product being
dimensionless.  There is considerable freedom of scaling available in choosing the functions 
$f_{C,i}$ and~$f_{E,i}$
and the elements making up the~$S_i(x)$.  Clearly, only the products $a_j s_{ji}$ are needed to
define~$S_i(x)$, so that the same results are obtained for $a_j^* := ca_j$ and $s_{ji}^* := c^{-1}s_{ji}$,
for any $c > 0$.  Similarly, if we had $S^*_i(x) := cS_i(x)$ for all $i$ and~$x$, we could choose
$f^*_{C,i}(s) := f_{C,i}(c^{-1}s)$ and~$f^*_{E,i}(s) := f_{E,i}(c^{-1}s)$.  The choice of the 
factor~$n^{-1}$ multiplying the sum in~\Ref{IFM:Connect} is made so that~$S_i(x)$ corresponds to an
average over~$n$ entries.  This is not a universal choice; for instance, the areas used by \citet{Hanski:94} 
correspond here to $n^{-1}a_i$, $1\le i\le n$.  Whatever scalings are used, it makes sense to
choose them such that the typical rate of change of state for an individual patch is neither very
small nor very large, as would presumably be to be expected in real situations.  The theorems that we
prove are, however, not sensitive to the particular choices made. The key requirement for keeping
the bounds small is that the overall number of changes of state expected per patch should be moderate.

\citet{OH:01} proposed a related deterministic model, analogous to~\Ref{Eq1} with $m=1$. Let $p_{i,t}$ be the probability that patch~$i$ is occupied at time~$t$ and let $ p_t = (p_{1,t},\ldots,p_{n,t}) $. As in the incidence function model, they model the change in $p_t$ by
\begin{equation}
   p_{i,t+1}-p_{i,t} \Eq C_{i}(p_t)(1-p_{i,t}) - E_{i}(p_t) p_{i,t}. \label{Eq2}
\end{equation}
They allow the probability of extinction at patch $ i $ to depend on the state of the whole metapopulation, in order to incorporate the rescue effect. We shall also consider the generalization of~\Ref{Eq2}, 
\begin{equation}
   p\um_{i,t+1}-p\um_{i,t} \Eq m^{-1}C_{i}(p\um_t)(1-p\um_{i,t}) - m^{-1}E_{i}(p\um_t) p\um_{i,t}, 
          \label{Eq2a}
\end{equation}
to mirror~\Ref{Eq1}.

\subsection{Spatially realistic Levins model}\label{Levins-model}

The spatially realistic Levins model \citep{HG:97} is the system of ordinary differential equations 
\begin{equation}
   \frac{dp_{i}(t)}{dt} \Eq C_{i}(p(t))(1-p_{i}(t)) - E_{i}(p(t)) p_{i}(t), \label{Eq3}
\end{equation}
for $p\colon [0,\infty) \to [0,1]^n$, where, as in model (\ref{Eq2}), $ C_{i}(p) = f_{C,i}(S_{i}(p)) $ and $ E_{i}(p) =  f_{E,i}(S_{i}(p)) $.
Although~$p(t)$ is meant to represent the probability that a patch in the metapopulation is occupied, the underlying stochastic model is unclear. 

We consider an appropriate stochastic version of model (\ref{Eq3}) to be the following generalization of the metapopulation model 
proposed by \citet[section 6.3]{AM:02}.  This model is a continuous time Markov chain $ X(t) = (X_{1}(t),\ldots,X_{n}(t)) $ on $ \xx$, where
\begin{eqnarray}
  \begin{array}{rcl}
  X \rightarrow X + \delta_{i}^{n} & \quad \mbox{at rate} \quad & C_{i}(X) (1-X_{i});  \\
  X \rightarrow X - \delta_{i}^{n} & \quad \mbox{at rate} \quad & E_{i}(X) X_{i}, 
  \end{array} \label{ADB-Eq4}
\end{eqnarray}
and $ \delta^{n}_{i}$  is the vector of length $ n $ with $ 1 $ at position $ i $ and zeros elsewhere. 

\subsection{Distance between models}\label{distance}

To discuss how well the deterministic models (\ref{Eq2}) and (\ref{Eq3}) approximate their corresponding stochastic models (\ref{Eq1}) and (\ref{ADB-Eq4}), we need a way to measure the closeness of the two models. For instance, we could consider comparing $ \ex X(t) $ from~\Ref{ADB-Eq4} with~$p(t)$ from~\Ref{Eq3}. However, we are typically interested in the behaviour of a given realisation of the metapopulation rather than its expectation. We thus prefer to compare the two metapopulations through the random measure valued processes $ (\bX(t),\,t\ge0)$ and $ (\bp(t),\,t\ge0) $ defined by 
\eq
   \begin{array}{rcl}
  \bX(t)\{B\} & := & n^{-1} \sum_{i=1}^{n} X_{i}(t)\, \mathbb{I}\left[(z_{i},a_{i}) \in B \right], \\
  \bp(t)\{B\} & := & n^{-1} \sum_{i=1}^{n} p_{i}(t)\, \mathbb{I}\left[(z_{i},a_{i}) \in B \right],
   \end{array} \label{ADB-measure-defs}
\en
for measurable sets $ B \subset \ww $. We say that the two models are close for $0\le t\le T$ if, for a suitable collection of measurable sets $ \mathcal{B} $,
\begin{equation}
  \sup_{0\le t\le T}\sup_{B \in \mathcal{B}}\left| \bX(t)\{B\} - \bp(t)\{B\} \right|  \label{BA:Eq1}
\end{equation}
is small with high probability. If (\ref{BA:Eq1}) is small, then the deterministic model provides a good approximation to the proportion of occupied patches in $ B $ relative to the entire metapopulation, for all $ B \in \mathcal{B} $. If we let $ \mathcal{B} $ be the Borel sets, then 
$$
\sup_{B \in \mathcal{B}}\left| \bX(t)\{B\} - \bp(t)\{B\} \right|
$$
is the total variation distance, and is given by
\begin{equation}
   \max \left( n^{-1} \sum_{i: X_{i}(t) = 1} \left(1 - p_{i}(t)\right),\ n^{-1} 
              \sum_{i: X_{i}(t) = 0} p_i(t) \right). \label{BA:Eq2}
\end{equation}
Although  $ \bX(t) $ and $ \bp(t) $  may not be close in total variation, it may still be possible for (\ref{BA:Eq1}) to be small, if we restrict the class of sets $ \mathcal{B} $. Specifically, we shall restrict the class of sets to those with finite Vapnik--Chervonenkis dimension.

\subsection{A brief summary of Vapnik--Chervonenkis theory}\label{VC-theory}

Vapnik--Chervonenkis theory concerns the uniform convergence of empirical measures over certain classes of sets. A central concept in Vapnik--Chervonenkis theory, and the part of the theory that  we will need in the following, is that of Vapnik--Chervonenkis (VC) dimension. 

The VC dimension is a measure of the size of a class of sets. Let $ \mathcal{B} $ be a class of sets in $ \mathbb{R}^{d} $.  To determine the VC dimension of $ \mathcal{B} $, we first need its  shatter coefficients which are defined by
$$
  S_{\mathcal{B}}(n) \Def \max_{x_{1},\ldots,x_{n} \in \mathbb{R}^{d}} \left| \left\{\{x_{1},\ldots,x_{n}\} \cap B; 
     B \in \mathcal{B} \right\} \right|,
$$
for $ n = 1,2,\ldots $ The shatter coefficient $ S_{\mathcal{B}}(n) $ is the maximal number of different subsets that can be formed  by intersecting a set of $ n $ points with elements of $ \mathcal{B} $. The VC dimension of a class of sets $ \mathcal{B} $ is the largest integer $ n $ such that  $ S_{\mathcal{B}}(n) = 2^{n} $. A corollary to a result of \citet{Sauer:72} shows that, for a class $ \mathcal{B} $ with VC dimension $ V $, the shatter coefficients can be bounded by $ S_{\mathcal{B}}(n) \leq (n+1)^{V} $ \citep[see][Corollary 4.1]{DL:01}. Examples of classes with finite VC dimension include the class of all rectangles in $ \mathbb{R}^{d} $ $ (V= 2d) $ and the class of closed balls in $ \mathbb{R}^{d} $ $(V = d + 1)$ \citep{Dudley:79}.

By restricting attention to the proportion of patches occupied within
each of the subsets of~$\ww$ that belong to a class of finite VC dimension, we are able to justify accurate approximation
of all the proportions simultaneously, whatever the underlying landscape.  Since, as illustrated
above, such classes of sets are very large, this should not be considered to be a major limitation of
the analysis. For instance, if the proportion of patches occupied in {\it every\/} rectangle in~$\ww$
is well approximated by its deterministic prediction, this constitutes a strong practical
justification for judging the deterministic approximation to be a good one.
Even when comparing the empirical measure from a sample of independent and identically distributed random variables to the true underlying probability measure, such a restriction is necessary \citep[Theorem 4]{VC:71}.

\section{Comparisons in discrete time}\label{discrete}

\subsection{Independent patches approximation}\label{IPA}

For a fixed $m\ge1$, define the process $ W_{t}\um = (W_{1,t}\um,\ldots, W_{n,t}\um) $ where, conditional on the environmental variables $  \sigma $, the $ W_{i,t}\um $ are independent Markov chains given by 
\begin{equation}
  \pr\left(W_{i,t+1}\um = 1 \ \middle|\ W_{i,t}\um \right) 
     \Eq m^{-1}C_{i}(p\um_t)(1-W_{i,t}\um) +  (1-m^{-1}E_{i}(p\um_t)) W_{i,t}\um, \label{IPA:Eq1}
\end{equation}
and~$p\um$ satisfies~\Ref{Eq2a} with $ p\um_{i,0} := \pr(W_{i,0}\um = 1) $.  Note that
\eq\label{ADB-tilde-X-mean}
      \ex ( W_{i,t}\um) \Eq  p\um_{i,t} \quad\mbox{for all}\ t. 
\en
Write
\begin{eqnarray*}
   \bW_t\um\{B\} & := & n^{-1} \sum_{i=1}^{n} 
             W_{i,t}\um \mathbb{I}\left[(z_{i},a_{i}) \in B \right];  \label{IPA:Eq1b} \\
   \bp_t\um\{B\} & := & n^{-1} \sum_{i=1}^{n} 
             p_{i,t}\um \mathbb{I}\left[(z_{i},a_{i}) \in B \right],    \label{IPA:Eq1c}
\end{eqnarray*}
for any measurable set $ B \subset \ww $. For the rest of this section, we suppress the superscript~$(m)$.


We begin by showing that $ \bW_t $ is well approximated by $ \bp_t $. For a measure~$\nu$ and function~$f$, define $\nu(f) := \int f\,d\nu$. The basic result concerns linear combinations of the form $ \bW_{t}(g) = \sn g_{in}\tX_{i,t}$, where $g_{in} := n^{-1}g(z_i,a_i)$ for $g\colon \ww \to \re$.

\begin{lemma} \label{ADB-IPA:Lem0}
For any $ \epsilon > 0 $,
\[
   \pr\Blb \Blm \bW_{t}(g) - \bp_{t}(g) \Brm > \e \Brb \Le 2\exp\{-2n\e^2/G_n^2\},
\]
where $G_n^2 := n\sn g_{in}^2 = n^{-1}\sn \{g(z_i,a_i)\}^2$.
\end{lemma}

\begin{proof}
The random variables $Y_i := g_{in}(\tX_{it} - p_{i,t})$, $1\le i\le n$, are independent, and $-g_{in} p_{i,t} \le Y_i \le g_{in}(1 - p_{i,t})$.  The lemma now follows from \citet[Theorem 2.5]{McDiarmid:98}.
\end{proof}

Applying the lemma with $g(w) := \mathbb{I}[w\in B]$, $w\in\ww$, for any $B \in \BB$ gives the following bound for classes~$\BB$ of sets.

\begin{corollary} \label{IPA:Lem1}
For any $ \epsilon > 0 $,
\begin{eqnarray*}
   \pr\left\{\sup_{B \in \mathcal{B}} \left| \bW_t\{B\} - \bp_t\{B\}  \right| 
           > \epsilon \right\} \Le 2 S_{\mathcal{B}}(n) \exp(-2n\epsilon^{2}).
\end{eqnarray*}
\end{corollary}

\begin{proof}
For any $ B $, let $ \xi_{t}\{B\} = \bW_t\{B\} - \bp_t\{B\}  $. Let $ \hat{\mathcal{B}} \subset \mathcal{B} $ denote a collection of sets such that any two sets in $ \hat{\mathcal{B}} $ have different intersections with the set 
$$ 
    \{(z_{1},a_{1}),\ldots,(z_{n},a_{n}) \}, 
$$ 
and every intersection is represented once. Then
\begin{eqnarray*}
  \pr\left\{\sup_{B \in \mathcal{B}} \left| \xi_{t}\{B\}  \right| > \epsilon  \right\}
     \Eq \pr\left\{\max_{B \in \hat{\mathcal{B}}} \left| \xi_{t}\{B\}  \right| > \epsilon  \right\}
     \Le \sum_{B \in \hat{\mathcal{B}}} \pr\left\{\left| \xi_{t}\{B\}  \right| > \epsilon  \right\} .
\end{eqnarray*}
But the final probability is of the form given in Lemma~\ref{ADB-IPA:Lem0}, with $g_{in} \in n^{-1}\{0,1\}$, giving $G_n^2 \le 1$, and hence
$$
  \pr\left\{\left| \xi_{t}\{B\}  \right| > \epsilon  \right\}  \Le 2 \exp(-2n\epsilon^{2}).
$$
To complete the proof, we simply note that $ \left| \hat{\mathcal{B}}\right| \leq S_{\mathcal{B}}(n) $.
\end{proof}

When $ \mathcal{B} $ has VC dimension $ V < \infty $, Corollary~\ref{IPA:Lem1} together with Sauer's (1972) bound $ S_{\mathcal{B}}(n) \leq (n+1)^{V} $ yields 
$$
      \pr\left\{\sup_{B \in \mathcal{B}} \left| \bW_t\{B\} - \bp_t\{B\}  \right| 
         > \left(\frac{C \log n}{n}\right)^{1/2}  \right\} \leq 2^{V+1} n^{V-2C}, 
            \label{IPA:Eq3}
$$
for any $ C>0 $. 

The following further consequence of Lemma~\ref{ADB-IPA:Lem0} is useful in the next section. We write
\eq\label{ADB-H-defs}
   \bbb_{in}^2 \Def  n^{-1}\sjn \{a_js_{ji}\}^2. 
\en

\begin{corollary} \label{ADB-IPA:Cor2}
Taking $g^{(i)}\colon \ww\to\re$ to be such that $g^{(i)}_{jn} := n^{-1}a_js_{ji}$, for any $1\le i\le n$, we have
\[
   \pr\Blb \Bigl| S_{i}(\tX_{t}) - S_i(p_t) \Bigr| > \e \Brb \Le 2\exp\{-2n\e^2/\bbb_{in}^2\}.
\]
\end{corollary}

Defining 
\eq\label{ADB-eps-def}
  \e_n(r) \Def n^{-1/2}\sqrt{r\log n}, 
\en
and letting 
\eq\label{ADB-F_n-def}
   F(r,T) \Def \Blb \max_{1\le i\le n}
     \max_{1\le t\le mT} \bbb_{in}^{-1}\Bigl| S_{i}(\tX_{t}) - S_i(p_t) \Bigr| \le \e_n(r) \Brb ,
\en
Corollary~\ref{ADB-IPA:Cor2} implies that, for any $T > 0$ such that $mT$ is an integer,
\eq\label{ADB-S-bnd}
    \pr(F^c(r,T)) \Le 2mTn^{-2r+1},
\en
where $ F^{c} $ is the complement of $ F $.

\subsection{Coupled metapopulation models}\label{coupled-models}

We now couple the independent patches meta\-population model~$W\um$ to the original metapopulation model~$X\um$, thus showing that the models defined in \Ref{Eq1} and~\Ref{Eq2a} indeed generate measure valued processes $(\bX_t\um,\,t \in \Z_+)$ and~$(\bp_t\um,\,t\in\Z_+)$ that are close over intervals of length~$mT$, uniformly in~$m$. Once again, we suppress the superscript~$(m)$ throughout the section. Let $ U_{i,t},\ i=1,\ldots,n,\ t=1,2,\ldots $ be an array of independent uniformly distributed random variables on $ [0,1] $. The incidence function model (\ref{Eq1}) and the independent patches model (\ref{IPA:Eq1}) can be realized together by
starting with $ X_{i,0} = W_{i,0}$, $1\le i\le n$, and then, for $t\ge0$, sequentially defining
\begin{equation}
    X_{i,t+1} \Eq (1-X_{i,t})\mathbb{I}(U_{i,t} \le m^{-1}C_{i}(X_{t})) 
              + X_{i,t} \mathbb{I}(U_{i,t} \leq 1-m^{-1}E_{i}(X_t)), \label{CM:Eq1}
\end{equation}
and
\begin{equation}
  W_{i,t+1} \Eq (1-W_{i,t})\mathbb{I}(U_{i,t} \leq m^{-1}C_{i}(p_t))
        + W_{i,t} \mathbb{I}(U_{i,t} \leq 1-m^{-1}E_{i}(p_t)), \label{CM:Eq2}
\end{equation}
for $1\le i\le n$.  Using this construction, we can subtract (\ref{CM:Eq2}) from (\ref{CM:Eq1}) to give
\eqa
   J_{i,t+1}  &\le&  J_{i,t} 
           + \left|\bI(U_{i,t} \le m^{-1}C_{i}(X_{t})) - \bI(U_{i,t} \leq m^{-1}C_{i}(p_t)) \right|
             \bI(X_{i,t} = 0) \non\\
  &&\qquad\mbox{} + \left|\bI(U_{i,t} \le m^{-1}E_{i}(X_{t})) - \bI(U_{i,t} \leq m^{-1}E_{i}(p_t)) \right|
             \bI(X_{i,t} = 1).   \label{ADB-X-diff-increment}
\ena
where 
\eq\label{ADB-J-def}
   J_{i,t} \Def \max_{1\le s\le t} \mathbb{I}(X_{i,s} \neq \tX_{i,s}).
\en
Thus, if the differences $m^{-1}|C_{i}(X_{t}) - C_{i}(p_t)|$ and $m^{-1}|E_{i}(X_{t}) - E_{i}(p_t)|$, $1\le i\le n$, are small for each~$t$ in some interval, it suggests that not too many components of $X$ and~$\tX$ will differ there.  The next lemma makes use of this idea; to state it, we introduce some further notation. 
We suppose that the functions $ f_{C,i} $ and~$f_{E,i}$ are Lipschitz continuous with Lipschitz constants 
$ L_i(C) $ and~$L_i(E)$, and we write 
\eq\label{ADB-A-defs}
  \begin{array}{ll}
     \ba \Def n^{-1}\sn a_i; &\qquad L_i \Def  L_i(C) +  L_i(E);\\
     A \Def n^{-1}\max_{1\le i\le n} \sjn a_j L_j s_{ji}; &\qquad H \Def n^{-1}\sn a_i L_i \bbb_{in}  ,  
  \end{array}
\en 
where $\bbb_{in}$ is as defined in~\Ref{ADB-H-defs}.

\begin{lemma} \label{ADB-CM:Lem1}
Assume that the $ f_{C,i} $ and~$f_{E,i}$ are Lipschitz continuous with Lipschitz constants  $ L_i(C) $ and~$L_i(E)$. Then, with the notation of \Ref{ADB-H-defs} and~\Ref{ADB-A-defs}, we have
$$
   \ex \left( \sn a_i J_{i,mt}\right) \Le  n^{1/2}(H/A)\exp\{At\}.
$$
\end{lemma}

\begin{proof}
Under the assumptions of the lemma, 
\eq\label{ADB-C-diff}
  m^{-1}|C_{i}(X_{t}) - C_{i}(p_t)| \Le m^{-1} L_i(C)\Blb |S_i(X_{t}) - S_i(\tX_{t})|
    + |S_i(\tX_{t}) - S_i(p_t)| \Brb,
\en
and
\eq\label{ADB-E-diff}
  m^{-1}|E_{i}(X_{t}) - E_{i}(p_t)| \Le m^{-1} L_i(E)\Blb |S_i(X_{t}) - S_i(\tX_{t})|
    + |S_i(\tX_{t}) - S_i(p_t)| \Brb.
\en
Now
\eq\label{ADB-S-diff-1}
    |S_i(X_{t}) - S_i(\tX_{t})| \Le n^{-1}\sjn a_j s_{ji} |X_{j,t} - \tX_{j,t}|
                                      \Le n^{-1}\sjn a_j s_{ji} J_{j,t},
\en
and, as the $ W_{i,t} $ are independent Bernoulli random variables, it follows from \Ref{ADB-tilde-X-mean} that $\ex\{S_i(\tX_{t}) - S_i(p_t)\} = 0$ and
\eq\label{ADB-S-diff-2}
     \var\{S_i(\tX_{t}) - S_i(p_t)\}
    \Eq n^{-2} \sjn a^{2}_{j} s_{ji}^2 p_{j}(t) (1-p_{j}(t)) \Le n^{-1}\bbb_{in}^2.
\en
From Jensen's inequality, $ \mathbb{E} \left| S_i(\tX_{t}) - S_i(p_t)\right| \leq n^{-1/2} \bbb_{in} $. Hence, writing $x_{i,t} := \ex J_{i,t}$, it follows from \Ref{ADB-X-diff-increment} and \Ref{ADB-C-diff}--\Ref{ADB-S-diff-2} that 
\eqa
    x_{i,t+1} &\le& x_{i,t} + m^{-1}L_i\Blb n^{-1}\sjn a_j s_{ji} x_{j,t} 
            + n^{-1/2}\bbb_{in} \Brb .
              \label{ADB-mean-recursion}
\ena
This in turn implies that
\eq
    \sjn a_j x_{j,t+1} \Le (1+m^{-1}A) \sjn a_j x_{j,t} + m^{-1}n^{1/2} H. \label{ADB-mean-recusion-2}
\en
By construction $ X_{i,0} = W_{i,0} $ so $ x_{i,0} = 0 $ for all $ i $. Iterating (\ref{ADB-mean-recusion-2}) gives
\[
    \sn a_i x_{i,t} \Le m^{-1}n^{1/2} H \sum_{k=0}^{t-1} (1+m^{-1}A)^{k} \Le (H/A) n^{1/2} \exp\{At/m\},
\]
proving the lemma. 
\end{proof}

Now define 
\eq\label{ADB-I-psi-def}
   I(\th) \Def \{i\colon\, a_i < \th\ba\};\qquad \ps(\th) \Def n^{-1}|I(\th)|,
\en
so that $a_i/(\th\ba) \ge 1$ for $i \notin I(\th)$. Then it follows immediately from Lemma~\ref{ADB-CM:Lem1} that, for any class of sets~$\BB$, and for any~$t \le mT$,
\eqs
     \sup_{B \in \mathcal{B}} \left|\bX_t\{B\} - \bW_t\{B\} \right| 
        &\le&  n^{-1} \sum_{i=1}^{n} \left| X_{i,t} - W_{i,t} \right| \non \\
     &\le&  (n\th\ba)^{-1} \sum_{i=1}^{n} a_i J_{i,mT}  + \ps(\th). \label{ADB-set-bnd}
\ens
Combining this bound with Markov's inequality yields, for any $y > 0$,
\eqa
  \lefteqn{\pr \left(\max_{1 \le t \le mT}\sup_{B \in \mathcal{B}} 
             \left|\bX_t\{B\} - \bW_t\{B\} \right| > \ps(\th) + y\right) } \non\\
  &&\Le \pr\Bl (n\th\ba)^{-1} \sum_{i=1}^{n} a_i J_{i,mT} > y \Br \label{ADB-mean-diff-1}\\
  &&\Le  \frac1{yn\th\ba}\, \ex\Blb \sum_{i=1}^{n} a_i  J_{i,mT} \Brb 
             \Le  \frac{H}{yA\ba\th}\, n^{-1/2}e^{AT}  . \label{ADB-mean-diff}
\ena
This has immediate consequences for uniform approximation over VC classes~$\BB$ of sets. Combining 
Corollary~\ref{IPA:Lem1} and~\Ref{ADB-mean-diff}, with $y = n^{-1/2+\h}\bbb e^{At}/(A\ba\th)$, we 
obtain the following result.

\begin{theorem} \label{Thm1}
Assume that $ f_{C,i} $ and~$f_{E,i}$ are Lipschitz continuous with Lipschitz constants $ L_i(C)$ 
and~$L_i(E)$. If $ \mathcal{B} $ has VC dimension $ V < \infty $, then, for any $ \th,\h > 0 $ and 
any $ T <\infty $, 
\eqs
  &&\pr \left\{\max_{1\leq t\leq mT}\sup_{B \in \mathcal{B}} \left| \bX_t\um\{B\} - \bp_t\um\{B\}  \right| 
                           > \ps(\th) + n^{-1/2+\h}\{(\bbb/A\ba)\th^{-1}e^{AT} + 1\} \right\} \\
  &&\qquad\qquad \Le 2mT(n+1)^V e^{-2n^{2\h}} +  n^{-\h}, 
\ens
where $\ba$, $A$ and~$\bbb$ are defined in~\Ref{ADB-A-defs}, and $\ps$ is as in~\Ref{ADB-I-psi-def}.
\end{theorem}

\nin In particular, for asymptotics as~$n$ increases, if the quantities $a_i/\ba$ are uniformly bounded 
away from zero, $\ps(\th_0) = 0$ for all~$n$, for some $\th_0 > 0$. Then, if also $A$, $\max_{1\le i\le n}L_i$ 
and~$\bbb$ are bounded and~$T$ is fixed, Theorem~\ref{Thm1} gives a bound of asymptotic order~$n^{-\h}$ 
for the probability that the measures of any of the sets of~$\BB$
differ by more than $n^{-1/2+\h}$ at any time before~$mT$, for any $0 < \h < 1/2$, provided at least 
that~$m = m_n$ does not grow faster than a polynomially in~$n$.  These conditions can be relaxed in 
many ways. For instance, if the function~$\ps$ is bounded for all~$n$ by a function~$\hps$ such that 
$\lim_{\th\to0}\hps(\th) = 0$, then the right hand side of Theorem~\ref{Thm1} can be made small for 
any $\h < 1/2$ by choosing $\th = \th_n \to 0$ suitably slowly, with the measures of sets 
in~$\BB$ differing by at most $\ps(\th_n) + n^{-1/2+\h}$.  Thus, if $\hps(\th) = \th^\b$, one can 
take $\h = (2+\b)/\{4(1+\b)\}$ and $\th_n = n^{-1/\{4(1+\b)\}}$, giving approximation with 
accuracy $2n^{-\b/\{4(1+\b)\}}$ with failure probability of order $n^{-1/4}$. 

For Theorem~\ref{Thm1} to give useful asymptotics, it is more or less essential that the product~$AT$ 
should remain bounded as~$n$ increases.  In biological terms, $A$ is related to the maximal rate at 
which a patch can become empty or be recolonized, though it is not a direct expression of that
quantity.  $AT$ can be thought of as a corresponding estimate of the number of colonization
or catastrophic events that can occur in a single patch over the length of time over which
the approximation is made.

\subsection{Refined approximation}\label{refined}

Under ideal asymptotic circumstances, in which the quantities $a_i/\ba$ are uniformly bounded away from zero and both $A$ and~$\bbb$ are bounded, the upper bound given in~\Ref{ADB-mean-diff} for the mean $\ell_1$-distance between $n^{-1}X\um$ and~$n^{-1}\tX\um$ is of asymptotic order~$O(n^{-1/2})$. Similarly, the measures of sets under $\bW\um$ and~$\bp\um$ are shown by Corollary~\ref{IPA:Lem1} to differ by at most order $O(n^{-1/2}\sqrt{\log n})$.  Using~\Ref{ADB-mean-diff} together with Markov's inequality thus shows that this is the right order for the differences between the  measures of sets under $\bX\um$ and~$\bp\um$, except on a set of probability of order $O(\{\log n\}^{-1/2})$.  Although this bound on the probability of the exceptional set converges to zero
as $n\to\infty$, it does so extremely slowly.  In this section,  a more complicated argument is used to show that the probability of the exceptional set is typically rather smaller.  Once more, we suppress the superscript~$(m)$.

The aim is to show that the $\ell_1$-distance between $n^{-1}X$ and~$n^{-1}\tX$ is of asymptotic order~$O(n^{-1/2})$, except on an event whose probability is also of order~$O(n^{-1/2})$.  To do this, we examine the process $J$ of~\Ref{ADB-J-def} in more detail. From \Ref{ADB-X-diff-increment}, on the set $ \{J_{i,t} = 0 \} $,
\eqs
J_{i,t+1} & \leq & \left| \mathbb{I} \left( U_{i,t} \leq m^{-1} C_{i}(X_{t})\right) - \mathbb{I} \left(U_{i,t} \leq m^{-1} C_{i}(p_{t})\right) \right| \\
& & + \left| \mathbb{I} \left( U_{i,t} \leq m^{-1} C_{i}(X_{t})\right) - \mathbb{I} \left(U_{i,t} \leq m^{-1} C_{i}(p_{t})\right) \right|
\ens
Recalling~\Ref{ADB-F_n-def}, it follows from \Ref{ADB-C-diff} and~\Ref{ADB-E-diff} that
\eqa
\lefteqn{\pr(J_{i,t+1} = 1 \giv \ff_t \cap \{J_{i,t} = 0\} \cap F(r,t/m))} \nonumber\\
 & \leq & m^{-1} L_{i} \ex \left(  |S_i(X_{t}) - S_i(\tX_{t})|
    + |S_i(\tX_{t}) - S_i(p_t)|\giv \ff_t \cap \{J_{i,t} = 0\} \cap F(r,t/m)\right), \nonumber \\
 \label{Reviewer1-point13-1}
\ena
where  $\ff_t $ is the sigma algebra generated by $ J_{i,s},\, 0\le s\le t, 1\le i\le n$ and denotes the history of~$J$ until time~$t$. Combining \Ref{ADB-S-diff-1} with \Ref{Reviewer1-point13-1} yields
\eqs
   \pr(J_{i,t+1} = 1 \giv \ff_t \cap \{J_{i,t} = 0\} \cap F(r,t/m))
     &\le& P_i(J_t),
\ens
where
\eq\label{ADB-P-def}
     P_i(J) \Def m^{-1}L_i \Blb n^{-1}\sjn a_j s_{ji} J_j + \bbb_{in}\e_n(r) \Brb.
\en
Furthermore, the $(J_{i,t+1},\,1\le i\le n)$ are conditionally independent, given~$\ff_t$. Hence, on the event $F(r,T)$, the process~$J$ is stochastically dominated for all times $1\le t\le mT$ by a process~$J^1 := (J_t^1,\,1\le t\le mT)$ on $\xx$,  which can be recursively determined from a collection $(U_{i,t,l},\,1\le i\le n,\,t,l\in\Z_+)$ of independent uniform random variables on $[0,1]$, together with the initial condition
$J_{i,0}^1 = 0$ for all~$i$, according to the prescription
\eq\label{ADB-J1-recursion}
   J_{i,t+1}^1 \Eq J_{i,t}^1 + \slo \bI(U_{i,t+1,l} \le P_{i}(J_t^1) - l).
\en 
Note that, typically, one would expect to have $P_i(J_t^1) \le 1$ , so that all but the zero term in the $l$-sum would be zero, but this need not be the case. Letting $Z_t := \sn a_i J_{i,t}^1$, and defining
\eq\label{ADB-A2-def}
   A_2 \Def  \max_{1\le j\le n} n^{-1} \sn a_i^2 L_i s_{ij};\qquad H_2 \Def n^{-1}\sn a_i^2 L_i \bbb_{in},  
\en
we have the following bounds on the first two moments of~$Z_{mt}$.

\begin{lemma}\label{ADB-variance}
 Assume that $ f_{C,i} $ and~$f_{E,i}$ are Lipschitz continuous with Lipschitz constants  $ L_i(C) $ and~$L_i(E)$. Then, with the notation of \Ref{ADB-H-defs}, \Ref{ADB-A-defs} and~\Ref{ADB-A2-def}, we have
$$
   \ex Z_{mt}   \Le  A^{-1}H n\e_n(r)e^{At}  ; \qquad
   \var Z_{mt} \Le A^{-2}(A_2H + H_2 A) n\e_n(r)e^{2At}. \label{ADB-variance-bnd}
$$
\end{lemma}

\begin{proof}
The formula for $\ex Z_{mt}$ follows as in the proof of Lemma~\ref{ADB-CM:Lem1}, but with $n^{-1/2}\bbb_{in}$ replaced by~$n\e_n(r)\bbb_{in}$ in~\Ref{ADB-mean-recursion}.  For the variance, it is immediate from~\Ref{ADB-J1-recursion} that
\[
    \var(Z_{t+1} \giv \ff^1_{t}) \Le \sn a_i^2 P_i(J_{t}^1) \Le m^{-1}A_2 Z_{t} + m^{-1}H_2 n\e_n(r),
\]
giving
\eq\label{ADB-variance-1}
    \ex\{\var(Z_{t+1} \giv \ff^1_{t})\} \Le m^{-1}\{A_2 \ex Z_{t} + H_2 n\e_n(r)\}.
\en
On the other hand, again from~\Ref{ADB-J1-recursion},
\eq\label{ADB-variance-2}
    \var\{\ex( Z_{t+1} \giv \ff^1_{t})\} \Eq \var\Blb \sjn (a_j + b_j) J_{j,t}^1 \Brb,
\en
where
\[
    b_j \Def m^{-1}\sn a_i L_i  n^{-1} a_j s_{ji} \Le m^{-1}A a_j.
\]
Since the $(J_{j,t}^1,\,1\le j\le n)$ are all decreasing functions of the independent random variables $(U_{i,s,l},\,1\le i\le n,\,s,l\in\Z_+)$, they are positively associated, implying that
\eq\label{ADB-variance-3}
    \var\Blb \sjn (a_j + b_j) J_{j,t}^1 \Brb \Le \var\Blb \sjn (1+m^{-1}A) a_j J_{j,t}^1 \Brb
     \Eq (1+m^{-1}A)^2 \var Z_{t}.
\en
Thus, from \Ref{ADB-variance-1} -- \Ref{ADB-variance-3}, it follows that
\[
    \var Z_{t+1} \Le (1+m^{-1}A)^2 \var Z_{t} + m^{-1}n\e_n(r)\{(A_2 H/A) \exp\{At/m\} + H_2 \}.
\]
Solving this recursion gives
\[
    \var Z_t \Le A^{-2}(A_2 H + H_2 A) n\e_n(r) \exp\{2At/m\},
\]
and the lemma is proved.
\end{proof}

As a direct result of Lemma~\ref{ADB-variance}, we have the following theorem.

\begin{theorem} \label{ADB-refined-Thm}
Assume that $ f_{C,i}$ and~$f_{E,i}$ are Lipschitz continuous with Lipschitz constants $ L_i(C)$ and~$L_i(E) $.
Suppose that we can choose $r \le n/\log n$ such that $\{2r-V-1\}\log n \ge \log (m/A)$. 
If $ \mathcal{B} $ has VC dimension $ V < \infty $, then, for any $ \th > 0 $ and  $ T <\infty $, 
\eqs
  &&\pr\left\{\max_{1\leq t\leq mT}\sup_{B \in \mathcal{B}} \left| \bX_t\um\{B\} - \bp_t\um\{B\}  \right| 
                           > \ps(\th) + \{2(\bbb/A\ba)\th^{-1}e^{AT}+1\}\e_n(r) \right\} \\
  &&\qquad \Le \frac{2AT}n + \frac{2^{V+1}AT}n
              +  \frac1{n\e_n(r)}\,\frac{A_2\bbb + H_2 A}{\bbb^2},\phantom{XXX}
\ens
where $\e_n(r)$ is defined in~\Ref{ADB-eps-def}, $\ba$, $A$ and~$\bbb$ in~\Ref{ADB-A-defs}, 
$A_2$ and~$H_2$ in~\Ref{ADB-A2-def}, and $\ps$ in~\Ref{ADB-I-psi-def}.
\end{theorem}

\begin{proof}
The conditions on $m$ and~$r$ ensure that $\pr[F^c(r,T)] \le 2ATn^{-1}$, using \Ref{ADB-S-bnd}, 
and that Corollary~\ref{IPA:Lem1} with
$\e = \e_n(r)$ gives a bound~$\g_n$ for the error probability satisfying $m\g_n \le 2^{V+1}An^{-1}$; 
they can clearly be satisfied for all~$n$ large enough, if~$m=m_n$ is such that $m_n/A$
grows at most like a fixed power 
of~$n$. The theorem now follows from  Corollary~\ref{IPA:Lem1}, \Ref{ADB-mean-diff-1} and 
Lemma~\ref{ADB-variance}, because, on~$F(r,T)$, $\bone^T J_t^1$ is an upper bound for $\bone^T J_t$. 
\end{proof}

The statement of Theorem~\ref{ADB-refined-Thm} can be illustrated by first considering a context in which the~$a_i$ are all equal to some value~$a$, the~$s_{ij}$ are all equal to~$1$, and the $L_i$ are all equal to some value~$L$; this represents a community of patches of equal merit where the distance between patches has no effect on the colonisation probabilities. Then $\ba = H_{in} = a$, $A = aL$, $H = A_2 = a^2L$ and $H_2 = a^3L$, so that
\[
     \frac H{A\ba} \Eq 1 \quad\mbox{and}\quad \frac{A_2\bbb + H_2 A}{\bbb^2}  \Eq 2.
\]
Thus, taking $\th =  1$, the error in approximating $\bX_t\um\{B\}$ by $\bp_t\um\{B\}$ is uniformly bounded for $B \in \BB$ by a quantity which grows exponentially in time~$T$ (corresponding to~$mT$ steps in the $m$-process), and is of order $O(n^{-1/2}\sqrt{\log n})$ as~$n$ increases; this bound is valid except on an event of probability of order~$O(n^{-1/2})$.  Suppose, instead, that for each $ i $, exactly $ d_{i} $ of the $ s_{ij} $ are equal to~$1$ and the rest are zero. Treating the metapopulation network as a graph, $ d_{i} $ is the degree of patch $ i $. Then 
\[
   \frac H{A\ba} \Eq n^{1/2} \left( \frac{n^{-1} \sum_{i=1}^{n} d_{i}^{1/2} }{\max_{1\leq i\leq n} d_{i}}\right) \quad\mbox{and}\quad \frac{A_2\bbb + H_2 A}{\bbb^2}  \Eq
              2n^{-1/2} \left( \frac{\max_{1\leq i\leq n} d_{i}}{n^{-1} \sum_{i=1}^{n} d_{i}^{1/2} }\right)  ,
\]
so the bound given in Theorem~\ref{ADB-refined-Thm} is determined by the maximal degree and a moment of the degree distribution.
In particular, if $ d_{i} = d(n) $ for all $ i $, then the  probability of the exceptional event given in Theorem~\ref{ADB-refined-Thm} is of smaller order than $O(n^{-1/2})$ if $d(n)/n \to 0$, but the bound on the differences between $\bX_t\um\{B\}$ and $\bp_t\um\{B\}$ is of larger order~$O(d(n)^{-1/2}\sqrt{\log n})$.

\section{Comparisons in continuous time}\label{continuous}

The arguments in the previous sections can also be applied to the spatially realistic Levins model. One approach is to use the results of the previous sections, and to consider the limit as $m \to \infty$.  More precisely, one can choose~$m = m_n$ so large that the continuous time random process is identical to a discrete time process on a close mesh of time points, except on an event of negligible probability.  Then, at least when the $L_i(C)$ and~$L_i(E)$ are uniformly bounded, the solution to the differential equations~\Ref{Eq3} can be shown for such~$m$ to be very close to the solution to the difference equations~\Ref{Eq2a}.  However, in order to prove a theorem in the
same generality as those in the previous section, showing that the measures $\bX(t)$ and~$\bp(t)$ defined in~\Ref{ADB-measure-defs} are uniformly close for $t \in [0,T]$, it is easier to argue directly. 

In order to show that the Markov process~$X$ defined in~\Ref{ADB-Eq4} is close to the solution~$p$ to the differential equations~\Ref{Eq3} with the same initial value, we proceed as before, using an intermediate approximation~$W$.  This is an inhomogeneous Markov process on~$\xx$, with time dependent transition rates
$$
  \begin{array}{rcl}
     W \ \to\ W + \d_i^n  &\mbox{at rate} & C_i(p(t))(1 - W_i); \\
     W \ \to\ W - \d_i^n  &\mbox{at rate} & E_i(p(t))W_i.
  \end{array}\label{ADB-W-def}
$$
We proceed in two steps, showing first that the measures $\bW(t)$ and~$\bp(t)$ are close for all $0\le t\le T$, when evaluated at the elements~$B$ of a VC-class~$\BB$, where
$$\label{ADB-W-bar-def}
  \bW(t)\{B\} \Def n^{-1} \sum_{i=1}^{n} W_{i}(t)\, \mathbb{I}\left[(z_{i},a_{i}) \in B \right].
$$
We then show that $W$ and~$X$ can be coupled in such a way that $n^{-1}\sn a_i|W_i(t) - X_i(t)|$ remains 
small for $0\le t\le T$, from which the closeness of $\bW(t)$ and~$\bX(t)$ for such~$t$ then follows as before. 
  
To formulate the theorem, we introduce
\[
    k(C,E) \Def \max_{1\le i\le n}\max_{x\in\xx}\max\{C_i(x),E_i(x)\},
\]
the maximum possible rate of change of state of an individual patch.

\begin{theorem}\label{ADB-Levins-Thm}
Assume that $ f_{C,i} $ and~$f_{E,i}$ are Lipschitz continuous with Lipschitz constants $ L_i(C)$ and~$L_i(E)$. 
Assume that $An^{-1} \le k(C,E) \le An^\a$ for some $\a< \infty$, and that $ \mathcal{B} $ has VC 
dimension $ V < \infty $.  Choose any
\eq
   2r\ >\ V+5+2\a + (V+1)(\log 2/\log n). \label{ADB-r-choice}
\en 
Then, for any $ \th,\h > 0 $ and any $ T <\infty $, 
\eqs
  &&\pr \left\{\sup_{0\leq t\leq T}\sup_{B \in \mathcal{B}} \left| \bX(t)\{B\} - \bp(t)\{B\}  \right| 
                           > \ps(\th) + 2n^{-1} + \e_n(r) + n^{-1/2+\h}\th^{-1}e^{AT}\right\} \\
  &&\qquad\qquad \Le  \frac{5(AT+1)}n + 
            \frac{\bbb}{A\ba}n^{-\h}\sqrt{r\log n} , 
\ens
and
\eqs
  &&\pr\left\{\sup_{0\leq t\leq T}\sup_{B \in \mathcal{B}} \left| \bX(t)\{B\} - \bp(t)\{B\} \right| 
                           > \ps(\th) + 2n^{-1} + \e_n(r) + 2\e_n(r)(\bbb/A\ba)\th^{-1}e^{AT} \right\} \\
  &&\qquad \Le  \frac{5(AT+1)}n
              +   \frac1{n\e_n(r)}\,  \frac{2A_2\bbb + A\bbb_2}{2\bbb^2} ,\phantom{XXX}
\ens
where $\e_n(r)$ is as defined in~\Ref{ADB-eps-def}, $\ba$, $A$ and~$\bbb$ in~\Ref{ADB-A-defs}, 
$A_2$ and~$H_2$ in~\Ref{ADB-A2-def}, and $\ps$ in~\Ref{ADB-I-psi-def}.
\end{theorem}

\begin{proof}
 For given initial condition, the linear equations
\eq\label{ADB-linear-w}
  \frac{dw_i}{dt} \Eq (1-w_i) C_i(p(t)) - w_i E_i(p(t)),
\en
with time dependent coefficients $C_i(p(t))$ and~$E_i(p(t))$, $1\le i\le n$, $t\ge0$, have a unique solution, 
giving $w(t) = p(t)$ for all~$t$ if $w(0) = p(0)$.  On the other hand, \Ref{ADB-linear-w} is satisfied by 
$w(t) := \ex\{W(t) \giv W(0) = p(0)\}$, so that $\ex W(t) = p(t)$ for all~$t$ if $W(0) = p(0)$. Since, 
for each~$t$, the~$(W_i(t),\,1\le i\le n)$ are independent Bernoulli random variables, we can apply 
Lemma~\ref{ADB-IPA:Lem0} to deduce that, for any~$t,\e > 0$,
\begin{eqnarray}\label{ADB-VC-approx}
   \pr\left\{\sup_{B \in \mathcal{B}} \left| \bW(t)\{B\} - \bp(t)\{B\}  \right| 
           > \e \right\} \Le 2 S_{\mathcal{B}}(n) \exp(-2n\e^{2}),
\end{eqnarray}
and also that, as for Corollary~\ref{ADB-IPA:Cor2},
\eq\label{ADB-S-bnd-1}
   \pr\Blb \Bigl| S_{i}(W(t)) - S_i(p(t)) \Bigr| > \e \Brb \Le 2\exp\{-2n\e^2/\bbb_{in}^2\}.
\en

Fix any~$T > 0$.  For $h=h_n > 0$, to be chosen later, set $t_j := jh$, $0\le j\le \lceil T/h \rceil$. Then
\eqs
   \sup_{0\le t\le T}\sup_{B\in\BB}|\bW(t)\{B\} - \bp(t)\{B\}| 
       & \leq & \max_{1\leq j \leq n} \sup_{t_{j-1}\leq s \leq t_{j-1}} \sup_{B \in \mathcal{B}} 
              |\bW(s)\{B\} - \bW(t_{j-1})\{B\} |   \\
   & &\mbox{} +  \max_{1\leq j \leq n} \sup_{t_{j-1}\leq s \leq t_{j-1}} \sup_{B \in \mathcal{B}} 
                  | \bp(s)\{B\} - \bp(t_{j-1})\{B\}| \\
   & & \mbox{} + \max_{1\leq j \leq n} \sup_{B \in \mathcal{B}} |\bW(t_{j-1})\{B\} - \bp(t_{j-1})\{B\}| .
\ens
The overall jump rate of the process~$W$ cannot exceed $nk(C,E)$, so that the probability that~$W$ makes more than 
one jump in one of the intervals $(t_{j-1},t_{j}]$, $1\le j\le \lceil T/h \rceil$, is at most $\lceil T/h \rceil \{nhk(C,E)\}^2 \le A(T+h)n^{-1}$ if $h_n \le n^{-3}A\{k(C,E)\}^{-2}$. Ensure this by taking $Ah_n = n^{-3-2\a}$.  So
\[
\pr \left( \max_{1\leq j \leq n} \sup_{t_{j-1}\leq s \leq t_{j-1}} \sup_{B \in \mathcal{B}} 
       |\bW(s)\{B\} - \bW(t_{j-1})\{B\} |  > n^{-1} \right) \leq (AT+1)n^{-1}.
\]
Then, on the other hand, because $|dp_i/dt| \le k(C,E)$ for all $i$ and~$t$, we have 
\[
    \sup_{t_{j-1} \le s,t \le t_j}\sn |p_i(s) - p_i(t)| \Le nhk(C,E),\quad 1\le j\le \lceil T/h \rceil,
\]
and this does not exceed~$n^{-1}$ for $h_n$ as above. From inequality (\ref{ADB-VC-approx}),
\[
\pr \left(\max_{1\leq j \leq n} \sup_{B \in \mathcal{B}}  |\bW(t_{j-1})\{B\} - \bp(t_{j-1})\{B\}| > \e_n(r) \right) 
         \leq 2 \lceil T/h_{n}\rceil S_{\BB}(n) \exp(-2n\e^{2}_{n}(r)). 
\]
Hence, for this choice of~$h_n$, and with $\e_n(r)$ as defined in~\Ref{ADB-eps-def}, for~$r$ as 
in~\Ref{ADB-r-choice}, so that $h_n^{-1}S_{\BB}(n)n^{-2r} \Le An^{-1}$, we have
\eq\label{ADB-first-approx}
  \pr\Bl \sup_{0\le t\le T}\sup_{B\in\BB}|\bW(t)\{B\} - \bp(t)\{B\}| > 2n^{-1} + \e_n(r) \Br
    \Le 3(AT+1)n^{-1}.
\en
Note also that, if~$W$ has at most one jump in each of the intervals $(t_{j-1},t_j]$, then, for 
$s \in (t_{j-1},t_j]$, $S_i(W(s))$ takes one of the values $S_i(W(t_{j-1}))$ or $S_i(W(t_j))$. 
Hence
\eqs
  \lefteqn{\sup_{t_{j-1} \leq s < t_{j}} | S_{i}(W(s)) - S_{i}(p(s))|} \\
  &&\Le \sup_{t_{j-1} \leq s \le t_{j}}\max\{|S_{i}(W(t_{j-1})) - S_{i}(p(s)) |,|S_{i}(W(t_{j})) - S_{i}(p(s)) |\}\\
  &&\Le \max\{|S_{i}(W(t_{j-1})) - S_{i}(p(t_{j-1}))|, |S_{i}(W(t_{j})) - S_{i}(p(t_{j}))|\} \\
  &&\qquad\mbox{}     +  \sup_{t_{j-1} \le s,t \le t_j} |S_{i}(p(s)) - S_{i}(p(t))|.
\ens
With the above choice of~$h_n$, again because $|dp_i/dt| \le k(C,E)$, 
\eqs
   |S_i(p(s)) - S_i(p(t))| &\le& h_n k(C,E) n^{-1}\sjn a_j s_{ji} \Le h_n k(C,E) \bbb_{in} \Le \e_n(r)\bbb_{in},
\ens
for any $i$ and any $s,t \in [t_{j-1},t_{j+1}]$, since $Ah_n k(C,E) = n^{-3-2\a}$. 
Therefore, for any $i$ and~$j$, 
\eqs
  \lefteqn{\sup_{t_{j-1} \leq s\leq t_{j}} |S_{i}(W(s)) - S_{i}(p(s)) | } \\
  &&\Le \max\{| S_{i}(W(t_{j})) - S_{i}(p(t_{j}))|, |S_{i}(W(t_{j-1})) - S_{i}(p(t_{j-1}))|\} + \e_n(r)\bbb_{in}, 
\ens
and hence, by~\Ref{ADB-S-bnd-1}, 
\eq\label{ADB-S-control}
   \pr\Bl \sup_{0\le t\le T}\max_{1\le i\le n} H_{in}^{-1}|S_i(W(t)) - S_i(p(t))| > 2\e_n(r) \Br
     \Le \frac{2n(T+h_n)}{n^{2r}h_n} \Le \frac{2(AT+1)}n,
\en
because~$r$ is also such that $h_n^{-1}n^{-2r+1} \Le An^{-1}$.

We now couple $W$ and~$X$, so as to remain close on $[0,T]$, as the components of a bivariate inhomogeneous 
Markov process $\{(W(t),X(t)),\,t\ge0\}$.  For any time~$t$ and any state $(w,x) \in \xx^2$ such that 
$w_i = x_i = 1$, the transition rates for jumps in the $i$-coordinates are given by
\eqs
   \begin{array}{rcll}
    (w,x) &\to& (w,x) - (\d_i^n,\d_i^n) &\atrate  \min\{E_i(p(t)),E_i(x)\};  \\
    (w,x) &\to& (w,x) - (\d_i^n,0)  &\atrate (E_i(p(t)) - E_i(x))_+; \\
    (w,x) &\to& (w,x) - (0,\d_i^n)  &\atrate (E_i(x) - E_i(p(t)))_+, 
   \end{array}\label{ADB-E-rates}
\ens
and the analogous expressions hold for $w_i = x_i = 0$.  For $(w_i,x_i) = (1,0)$, the rates are
\eqs
   (w,x) &\to& (w,x) - (\d_i^n,0)  \atrate E_i(p(t))\\
    (w,x) &\to& (w,x) + (0,\d_i^n)  \atrate C_i(x),
\ens 
and the analogous expressions hold for $(w_i,x_i) = (0,1)$; initially, $W(0) = X(0) \in \xx$. Using a 
similar calculation to \citet[Section 5]{BR:58}, the marginal processes $X$ and~$W$ are seen to be  
Markov chains with the desired transition rates. 

Define~$J(t) \in \xx$ by
\eq\label{ADB-J-def-cts}
   J_i(t) \Def 1 - \mathbb{I}\left[W_i(s) = X_i(s),\,0\le s\le t\right],
\en
and set $Z(t) := \sn a_i J_i(t)$; for $(t,w,x,J) \in \re_+\times\xx^3$, define
$$
  \begin{array}{rcl}
   F(t,w,x,J) &:=& \sn a_i(1-J_i)\{(1-w_i)|C_i(x) - C_i(p(t))| + w_i|E_i(x) - E_i(p(t))|\};\\
   G(t,w,x,J) &:=& \sn a_i^2(1-J_i)\{(1-w_i)|C_i(x) - C_i(p(t))| + w_i|E_i(x) - E_i(p(t))|\}.
  \end{array} \label{ADB-FG-def}
$$
Now $Z(t)e^{-At}$ is a function of the inhomogeneous Markov process
$\{(W(t),X(t),J(t)),$
$t\ge0\}$.  Because $W_i(t) = X_i(t)$ whenever $J_i(t) = 0$,
$Z(t)e^{-At}$ has infinitesimal drift and covariance given by
\[
    e^{-At}\{F(t,W(t),X(t),J(t)) - AZ(t)\}\quad\mbox{and}\quad e^{-2At}G(t,W(t),X(t),J(t))
\]
respectively.  Dynkin's formula then implies that
\[
   M(t) \Def Z(t)e^{-At} - \int_0^t e^{-As}\{F(s,W(s),X(s),J(s)) - AZ(s)\}\,ds 
\]
is a martingale, with predictable quadratic variation 
\eq\label{ADB-QV}
   \langle M \rangle_t \Def \int_0^t e^{-2As}G(s,W(s),X(s),J(s))\,ds.  
\en
Define the stopping time
$$\label{ADB-tau-def}
   \t_n(r) \Def \inf\{t\ge0\colon\, \max_{1\le i\le n} \bbb_{in}^{-1}|S_i(W(t)) - S_i(p(t))| \ge 3\e_n(r)\},
$$
and set $\t_n(r,t) := \min\{t,\t_n(r)\}$.  Then, using \Ref{ADB-C-diff} and~\Ref{ADB-E-diff} as for~\Ref{ADB-P-def}, we have, for $s \le \t_n(r)$, 
\eqa
   F(s,W(s),X(s),J(s)) 
    &\le& \sn a_i L_i (1 - J_i(s))\Blb n^{-1}\sjn a_j s_{ji} J_j(s) + \bbb_{in}\e_n(r) \Brb \non\\
    &\le& A Z(s) + n\e_n(r)\bbb  \label{ADB-Z-mean-growth}
\ena
and
\eqa
  G(s,W(s),X(s),J(s)) 
    &\le& \sn a_i^2 L_i (1 - J_i(s))\Blb n^{-1}\sjn a_j s_{ji} J_j(s) + \bbb_{in}\e_n(r) \Brb \non\\
    &\le& A_2 Z(s) + n\e_n(r)\bbb_2. \label{ADB-Z-variance-growth}
\ena
It thus follows from~\Ref{ADB-Z-mean-growth} and the optional sampling theorem that
\eqs
   e^{-At}\ex Z(\t_n(r,t)) &=& \ex \Blb \int_0^{\t_n(r,t)} e^{-As}\{F(s,W(s),X(s),J(s)) - AZ(s)\}\,ds \Brb \\
           &\le& \int_0^t e^{-As} n\e_n(r)\bbb\,ds \Eq A^{-1}n\e_n(r)\bbb (1 - e^{-At}),
\ens
and hence that
\eq\label{ADB-Z-mean-cts}
     \ex Z(\t_n(r,t)) \Le A^{-1}n\e_n(r)\bbb (e^{At} - 1).
\en
Then, by a similar argument,
\eqs
   e^{-At} Z(\t_n(r,t))  &\le& \int_0^t e^{-As} n\e_n(r)\bbb\,ds + |M(\t_n(r,t))|,
\ens
giving
\[
    \pr[Z(\t_n(r,T)) > 2A^{-1}n\e_n(r)\bbb e^{At}] 
       \Le \pr[|M(\t_n(r,T))| > A^{-1}n\e_n(r)\bbb].
\]
The process $ M^{2} - \langle M \rangle $ is a martingale. Applying the optional sampling theorem again with \Ref{ADB-QV}, \Ref{ADB-Z-variance-growth} and~\Ref{ADB-Z-mean-cts} gives 
\eqs
    \var\{M(\t_n(r,T))\} &=& \ex \{\langle M \rangle_{\t_n(r,T)} \}
               \Le \int_0^T e^{-2As}\{A_2 \ex Z(\t_n(r,s)) + n\e_n(r)\bbb_2\}\,ds \\
    &\le& n\e_n(r) \frac{2A_2\bbb + A\bbb_2}{2A^2},
\ens
so that, by Chebyshev's inequality,
\[
   \pr[|M(\t_n(r,T))| > A^{-1}n\e_n(r)\bbb] 
              \Le \frac1{n\e_n(r)}  \frac{2A_2\bbb + A\bbb_2}{\bbb^2}.
\]
Since $\pr[\t_n(r,T) < T] \le 2(T+1)n^{-1}$ by~\Ref{ADB-S-control}, it follows that
\eq\label{ADB-Z-key-bnd}
   \pr[Z(T) > 2A^{-1}n\e_n(r)\bbb e^{AT}] \Le 3(T+1)n^{-1} + \frac1{n\e_n(r)}  \frac{2A_2\bbb + A\bbb_2}{\bbb^2}.
\en
The theorem is now proved from \Ref{ADB-Z-mean-cts}, \Ref{ADB-S-control} and~\Ref{ADB-Z-key-bnd}, in the same way as Theorems \ref{Thm1} and~\ref{ADB-refined-Thm} were completed. 
\end{proof}

\section{Discussion}\label{discussion}

The theorems proved in Sections \ref{discrete} and~\ref{continuous} give explicitly computable measures of the differences between the predictions of a number of stochastic metapopulation models and their deterministic counterparts.  No assumptions about asymptotic behaviour as the number~$n$ of patches tends to infinity are needed.  However, in order to get an idea about when the approximations are good, it is useful to think in terms of asymptotics.

The precision of the approximation of $\bX\{B\}$ by $\bp\{B\}$ depends on the time interval~$T$ through the factor~$e^{AT}$, and, as already discussed, it is thus important for good approximation that the product~$AT$ should not be large. The other key factor is $\bbb/(A\ba)$.  Taking the case when the~$L_i$ are all equal, the ratio~$\bbb/\ba$ represents an average of the quantities~$\bbb_{in}$.
Now, if the probabilities $\pr[W_j(t) = 1]$ are bounded away from $0$ and~$1$, the `signal to noise' ratio $\sqrt{\var(S_i(W))}/\ex S_i(W)$ is given by
\eqs
   \Blb \sjn p_j(1-p_j) \{ n^{-1} a_j s_{ji}\}^2\Brb^{1/2} \Big/ \Blb n^{-1}\sum_{l=1}^n p_l a_l s_{li}\Brb
    &\asymp& n^{-1/2}\bbb_{in}\Big/ \Blb n^{-1}\sum_{l=1}^n  a_l s_{li}\Brb. 
\ens
If the values of $n^{-1}\sum_{l=1}^n  a_l s_{li}$ are all of size comparable to their maximum~$A$, it follows that $n^{-1/2}\bbb/(A\ba)$ represents an average of these `signal to noise' ratios, and its being small reflects situations in which the quantities $S_i(W)$ do not fluctuate much, as is the key to the approximation of $\bW$ by~$\bp$.  In Theorems \ref{ADB-refined-Thm} and~\ref{ADB-Levins-Thm}, the precision is principally expressed in terms of $\e_n(r)\bbb/(A\ba)$, which is asymptotically larger than $n^{-1/2}\bbb/(A\ba)$ only by the factor $\sqrt{r\log n}$.  Thus, the two theorems attain an almost optimal asymptotic precision.

In practical terms, the `signal to noise' ratio of $S_i(W)$ is small when the influence on patch~$i$ 
is made up of contributions from a large number of patches.  If this is not the case, our theorems 
do not indicate that the approximation of $\bX$ by~$\bp$ need be good, even for large~$n$.  
The example of the contact process on the sites $\{1,2,\ldots,n\}$ \citep{DL:88} shows that the 
approximation may indeed be very bad.  In this model, a Levins model~\Ref{ADB-Eq4}, $s_{ij} = 1$ if $|i-j| = 1$,  
and $s_{1n}=1$ also; otherwise, $s_{ij}=0$. All the~$a_i$ are equal, $C_i(x) = \lambda(x_{i-1} + x_{i+1})$, 
with $x_0 := x_n$ and $x_{n+1} := x_1$, and $E_i(x) = 1$.  The quantity $n^{-1/2}\bbb/(A\ba)$ takes the value  
$1/\sqrt2$, which does not become small as~$n$ increases.  When $\lambda > 1/2$, the differential 
equations~\Ref{Eq3} have extinction ($x_i = 0$ for all~$i$) as an unstable equilibrium, and an 
equilibrium with $x_i = 1 - 1/2\lambda$ for all~$i$ which is locally stable.  On the other hand, 
the stochastic process~\Ref{ADB-Eq4} becomes extinct in time of order $O(\log n)$, the same order 
as for the (pure death) process with $\l=0$, whenever $\l < \l_c$ \citep[Theorem 1]{DL:88}, 
where~$\l_c$ is the critical value for the same process on the whole of~$\Z$.  Since $3/2 < \l_c < 2$, 
the behaviour of the stochastic process~\Ref{ADB-Eq4} is completely different from that of its 
deterministic counterpart~\Ref{Eq3} when $1/2 < \l < 3/2$.  
 
In the context of habitat fragmentation, the condition that $ A $ remains bounded as $ n $ increases 
is natural. First, we note that $ s_{ji} \leq 1 $ for any of the forms considered in \citet{MH:98} 
and \citet{Moilanen:04}.  Comparing equation (\ref{IFM:Connect}) with the original formulation 
of \citet{Hanski:94}, we see that the area of patch $ i $ is given by $ n^{-1} a_{i} $. If we 
consider that the original habitable area was finite and that the habitat patches were formed 
by fragmentation of this area, then this implies that $ \bar{a} $ remains bounded. Assuming 
the $ L_{i} $ are bounded, $ A $ will also remain bounded. The other factor controlling the 
accuracy of the approximation, $\bbb/(A\ba)$, is also constrained in the habitat fragmentation 
context. If $ L_{i} \leq L $ for all  $i $, then 
$ H \leq L \bar{a} (n^{-1} \sum_{j=1}^{n} a_{j}^{2})^{1/2}$. If the area of the largest patch 
is bounded by $ \delta_{n} $, then 
$ n^{-1} \sum_{j=1}^{n} a_{j}^{2} \leq n \delta_{n} (\bar{a}  + 2\delta_{n}) $. Hence, 
$n^{-1/2}\bbb/(A\ba)=O(\delta_{n}^{1/2}) $. Therefore, the deterministic process provides a good 
approximation provided $ \max_{i} n^{-1} a_{i} \rightarrow 0 $. In other words, the area of the 
largest patch should be small for the approximation to be good. If one of more patches were to 
remain large, then we would expect the approximation to be poor. An example of the type of behaviour 
to be expected in this case is given in \citet{MP:12}.

Another natural asymptotic framework is that in which the area under consideration is taken to be 
progressively larger, encompassing ever more patches, but without the overall patch structure 
changing. In such circumstances, the numbers of patches influencing a given patch would not 
typically change with~$n$, and hence no improvement in precision is to be expected as~$n$ increases.  
The contact process discussed above is an example of this. 

\citet{OC:06} studied a similar problem, but allowed the number of patches influencing a given patch to 
increase by scaling the $ s_{ij}$. Their aim was to analyse how the stochastic and deterministic spatial 
models deviate from the simpler Levins model. In the simplest case, \citet{OC:06} assumed that the location 
of patches followed a Poisson process on $ \mathbb{R}^{d} $. To bring our analysis closer to theirs, assume 
that, in a metapopulation of $ n $ patches, the patch locations~$z_i$ are independent and uniformly distributed on 
$ [0,n^{1/d}]^d$. As $ n \rightarrow \infty $, the distribution of patches on any fixed finite region converges to 
that of a Poisson process.  With a constant rate of local extinction and colonisation 
function $ f_{C,i}(x) = x $ for all $ i $, it follows that $ L_{i} = 1 $ for all $ i $. To simplify the calculations, we assume that all patch areas are the same, and that interaction
occurs with the same intensity between all close enough patches.  Explicitly,
following the standardization in \citet{Hanski:94}, we choose $n^{-1}a_i = 1$ for all~$i$, 
and assume that
\[
   s_{ij} \Eq (v(d)R^d)^{-1} \mathbb{I} \left(|z_{i} - z_{j}| \leq R \right),
\]
where $R=R_n$ controls the range of influence of a patch, and $v(d)$ denotes the volume of the unit ball
$B_1(0)$ in~$\re^d$. \citet{OC:06} proposed expansions for the equilibrium level 
of the metapopulation that became more accurate in the limit as $R \rightarrow \infty $. To apply 
Theorem~\ref{ADB-Levins-Thm} to this setting, we need to 
calculate parameters such as $ \ba, A $ and $ H $.

It is immediate from our definitions that $\ba=n$, and that we can take $\th=1$ with $\ps(1) = 0$.  
The values of the remaining parameters depend on the positions
of the~$z_i$.  However, for each fixed~$i$, conditioning on the position~$z_i$, the sum
$\sum_{j\ne i} \mathbb{I} (|z_{i}-z_{j}| \leq R)$ has the binomial distribution $\Bi(n-1,p_{ni})$,
with $p_{ni} := n^{-1}|B_R(z_i) \cap [0,n^{1/d}]^d|$.  By the upper Chernoff inequality, it follows that,
for any $\e > 0$, if $ R^{d}/\log n \rightarrow \infty $, then
\begin{eqnarray*}
\lefteqn{\pr\left( \max_{1\le i\le n}\sum_{j\ne i} \mathbb{I} (|z_{i}-z_{j}| \leq R) \ge (1+\e)v(d)R^d \Br} \\
&  \Le  &    n\pr\left( \sum_{j\ne i} \mathbb{I} (|z_{i}-z_{j}| \leq R) \ge (1+\e)v(d)R^d \Br \ \to\ 0
\end{eqnarray*}
as $n\to\infty$.  If $R^d/n \to 0$, with probability tending to~$1$, one of the~$z_i$ is such that
$p_{ni} = n^{-1}v(d)R^d$, and it then follows also that
$$
    \pr\Bigg( \max_{1\le i\le n}\sum_{j\ne i} \mathbb{I} (|z_{i}-z_{j}| \leq R) \le (1-\e)v(d)R^d\Bigg)
      \ \to\ 0.
$$
Hence, if $\log n \ll R^d \ll n$, $A \in [1-\e,1+\e]$ with high probability, and we also have
$H = O(n^{3/2}R^{-d/2})$. Applying the first
 part of Theorem~\ref{ADB-Levins-Thm}, we see that $ \bX $ and $ \bp $ are close with high probability 
on the interval $[0,T]$, for any fixed~$T$,
if $ n^\d \ll R^d \ll n$, for any $0 < \d < 1$. 
For the second part of the theorem, we have $H/A\ba = O(\sqrt{n/R^d})$ and $\ps(1)=0$ as above, 
and, in addition, $(A_2H + AH_2)/H^2 = O(\sqrt{R^d/n})$. This gives an approximation error
of order $O(\sqrt{\log n/R^d})$ over any fixed interval $[0,T]$, 
uniformly for all sets in any class with finite VC dimension, except on an event of 
probability $O(n^{-1})$, thus sharpening the bound on the error probability, while
broadening the range of~$R$ to $\log n \ll R^d \ll n$.
The same result is true also if $R^d \asymp n$, though the value of~$A$ may be
different.

However, although we have close agreement between 
deterministic and stochastic models using a scaling similar to \citet{OC:06}, our results do not allow 
us to make similar statements. A crucial part of their analysis involved examining the behaviour of 
the equilibrium of deterministic model under the scaling of the colonisation kernel. Examining the 
behaviour of the deterministic model under this scaling for finite metapopulations would be an 
interesting problem for future study.

Distance between the measures $\bX$ and~$\bp$ has been described by bounding the differences between the probabilities that they assign to the sets in a class~$\BB$ of finite VC dimension. The assumption of a finite VC dimension reduces the number of integrals that need to be compared to a finite number that grows like a polynomial in $ n$.  However, one could look instead at other distances for which the number of integrals that needs to be compared grows faster than a polynomial in $n $, at the cost of losing some precision.  For instance, if such a distance requires $\exp\{\a n^\h\}$ integrals to be compared, with $\a > 0$ and $0 < \h < 1$, then this number is heavily dominated by the failure probability $\exp\{-n\e^2\}$ that follows, as for Corollary~\ref{IPA:Lem1}, from Lemma~\ref{ADB-IPA:Lem0}, if~$\e = \e_n$ is chosen to be $b n^{-(1-\h)/2}$ with $b^2 = 2\a$.  Thus the approximation of $\bW$ by~$\bp$ to this accuracy can be achieved for sufficiently many time points, with negligible probability of failure, and the approximation of $\bX$ by~$\bW$ is proved as before.  One example would be to use the Wasserstein distance between measures, assuming that the values $(z_{i},a_{i})$ come from a bounded subset $\ww_0$ of~$\ww$.  For instance, if~$\ww$ has dimension $d+1$, then the number of functions with
Lipschitz constant at most~$k_n$ needed to approximate any such function on~$\ww_0$ to within~$\e_n$ in supremum distance is of order $O(\exp\{\a(k_n/\e_n)^{d+1}\})$ for some $\a > 0$ \citep[section 5.1.1]{Lorentz:66} and taking $\e_n = b(k_n^{d+1}/n)^{1/(d+3)}$ with $b^{(d+3)} = 2\a$ would result in the difference between the expectations of any
Lipschitz functions with constant less than~$k_n$ being at most of order~$\e_n$, with negligible failure probability, if $k_n \le n^\h$ with $\h(d+1) < 1$.  For Wasserstein distance, we choose $k_n = 1$, and the distance is of order $O(n^{-1/(d+3)})$.

\section*{Acknowledgements}
We would like to thank the two referees for their helpful comments and suggestions.

\end{document}